\documentclass{article}

\title{\textbf{A symmetric $\mathbf{\beta}$-model}}

\author{Stephen G. Simpson\footnote{The research presented here was
    supported by NSF grant DMS-0070718.  This paper was written in the
    year 2000 but was not published until 2018.  The only change since
    2000 is that the references have been updated.}\\
  Department of Mathematics\\
  Vanderbilt University\\
  \href{http://www.math.psu.edu/simpson}{http://www.math.psu.edu/simpson}\\
  \href{mailto:sgslogic@gmail.com}{sgslogic@gmail.com}}
  
\date{First draft: April 14, 2000\\
  This draft: \today}

\usepackage[backref=page,colorlinks=true,allcolors=blue]{hyperref}

\usepackage{amssymb}
\newcommand{\HYP}{\mathrm{HYP}}
\newcommand{\REC}{\mathrm{REC}}
\newcommand{\Ltwo}{L_2}
\newcommand{\supp}{\mathrm{supp}}
\newcommand{\field}{\mathrm{field}}
\newcommand{\lh}{\mathrm{lh}}
\newcommand{\forces}{\Vdash}
\newcommand{\proves}{\vdash}
\newcommand{\PP}{\mathcal{P}}
\newcommand{\D}{\mathcal{D}}
\newcommand{\WKLo}{\mathsf{WKL}_0}
\newcommand{\ATRo}{\mathsf{ATR}_0}
\newcommand{\TIo}{\mathsf{TI}_0}

\usepackage{amsthm}
\theoremstyle{definition}
\newtheorem{thm}{Theorem}[section]
\newtheorem{que}[thm]{Question}
\newtheorem{rem}[thm]{Remark}
\newtheorem{cor}[thm]{Corollary}
\newtheorem{lem}[thm]{Lemma}

\begin{document}

\maketitle

\addcontentsline{toc}{section}{Abstract}

\begin{abstract}
  We prove that there exists a countable $\beta$-model in which, for
  all reals $X$ and $Y$, $X$ is definable from $Y$ if and only $X$ is
  hyperarithmetical in $Y$.  We also obtain some related results and
  pose some related questions.
\end{abstract}

\tableofcontents

\newpage

\begin{center}

\end{center}

\section{The main results}
\label{sec:main}

Our context is the study of $\omega$-models of subsystems of second
order arithmetic \cite[Chapter VIII]{sosoa}.  As in \cite[Chapter
VII]{sosoa}, a $\beta$\emph{-model} is an $\omega$-model $M$ such
that, for all $\Sigma^1_1$ sentences $\varphi$ with parameters from
$M$, $\varphi$ is true if and only if $M\models\varphi$.  Theorems
\ref{thm:betadef} and \ref{thm:betadef2} below are an interesting
supplement to the results on $\beta$-models which have been presented
in Simpson \cite[\S\S VII.2 and VIII.6]{sosoa}.

Let $\HYP$ denote the set of hyperarithmetical reals.  It is well
known that, for any $\beta$-model $M$, $\HYP$ is properly included in
$M$, and each $X\in\HYP$ is definable in $M$.

\begin{thm}\label{thm:betadef}
  There exists a countable $\beta$-model satisfying
  \begin{center}
    $\forall X\,($if $X$ is definable, then $X\in\HYP)$.
  \end{center}
\end{thm}

\begin{proof}
  Fix a recursive enumeration $S_e$, $e\in\omega$, of the $\Sigma^1_1$
  sets of reals.  If $p$ is a finite subset of
  $\omega\times\omega^{<\omega}$, say that $\langle
  X_n\rangle_{n\in\omega}$ \emph{meets} $p$ if
  $X_{n_1}\oplus\cdots\oplus X_{n_k}\in S_e$ for all $(e,\langle
  n_1,\dots,n_k\rangle)\in p$.  Let $\PP$ be the set of $p$ such that
  there exists $\langle X_n\rangle_{n\in\omega}$ meeting $p$.  Put
  $p\le q$ if and only if $p\supseteq q$.  Say that $\D\subseteq\PP$
  is \emph{dense} if for all $p\in\PP$ there exists $q\in\D$ such that
  $q\le p$.  Say that $\D$ is \emph{definable} if it is definable over
  the $\omega$-model $\HYP$, i.e., arithmetical in the complete
  $\Pi^1_1$ subset of $\omega$.  Say that $\langle
  G_n\rangle_{n\in\omega}$ is \emph{generic} if for every dense
  definable $\D\subseteq\PP$ there exists $p\in\D$ such that $\langle
  G_n\rangle_{n\in\omega}$ meets $p$.  We can show that for every
  $p\in\PP$ there exists a generic $\langle G_n\rangle_{n\in\omega}$
  meeting $p$.  (This is a fusion argument, a la Gandy forcing.)
  Clearly $\{G_n:n\in\omega\}$ is a $\beta$-model.  We can also show
  that, if $C$ is countable and $C\cap\HYP=\emptyset$, then there
  exists a generic $\langle G_n\rangle_{n\in\omega}$ such that
  $C\cap\{G_n:n\in\omega\}=\emptyset$.
  
  Let $\Ltwo(\langle X_n\rangle_{n\in\omega})$ be the language of
  second order arithmetic with additional set constants $X_n$, $n\in
  \omega$.  Let $\varphi$ be a sentence of $\Ltwo(\langle
  X_n\rangle_{n\in\omega})$.  We say that $p$ \emph{forces} $\varphi$,
  written $p\forces\varphi$, if for all generic $\langle
  G_n\rangle_{n\in\omega}$ meeting $p$, the $\beta$-model
  $\{G_n:n\in\omega\}$ satisfies $\varphi[\langle
  X_n/G_n\rangle_{n\in\omega}]$.  It can be shown that, for all
  generic $\langle G_n\rangle_{n\in\omega}$, the $\beta$-model
  $\{G_n:n\in\omega\}$ satisfies $\varphi[\langle
  X_n/G_n\rangle_{n\in\omega}]$ if and only if $\langle
  G_n\rangle_{n\in\omega}$ meets some $p$ such that $p\forces\varphi$.
  
  If $\pi$ is a permutation of $\omega$, define an action of $\pi$ on
  $\PP$ and $\Ltwo(\langle X_n\rangle_{n\in\omega})$ by
  $\pi(p)=\{(e,\langle\pi(n_1),\dots,\pi(n_k)\rangle):(e,\langle
  n_1,\dots,n_k\rangle)\in p\}$ and $\pi(X_n)=X_{\pi(n)}$.  It is
  straightforward to show that $p\forces\varphi$ if and only if
  $\pi(p)\forces\pi(\varphi)$.  The \emph{support} of $p\in\PP$ is
  defined by $\supp(p)=\bigcup\{\{n_1,\dots,n_k\}:(e,\langle
  n_1,\dots,n_k\rangle)\in p\}$.  Clearly if $p,q\in\PP$ and
  $\supp(p)\cap\supp(q)=\emptyset$, then $p\cup q\in\PP$.
  
  We claim that if $\langle G_n\rangle_{n\in\omega}$ and $\langle
  G'_n\rangle_{n\in\omega}$ are generic, then the $\beta$-models
  $\{G_n:n\in\omega\}$ and $\{G'_n:n\in\omega\}$ satisfy the same
  $\Ltwo$-sentences.  Suppose not.  Then for some $p,q\in\PP$ we have
  $p\forces\varphi$ and $q\forces\lnot\,\varphi$, for some
  $\Ltwo$-sentence $\varphi$.  Let $\pi$ be a permutation of $\omega$
  such that $\supp(\pi(p))\cap\supp(q)=\emptyset$.  Since
  $\pi(\varphi)=\varphi$, we have $\pi(p)\forces\varphi$, hence
  $\pi(p)\cup q\forces\varphi$, a contradiction.  This proves our
  claim.
  
  Finally, let $M=\{G_n:n\in\omega\}$ where $\langle
  G_n\rangle_{n\in\omega}$ is generic.  Suppose $A\in M$ is definable
  in $M$.  Let $\langle G'_n\rangle_{n\in\omega}$ be generic such that
  $M'=\{G'_n:n\in\omega\}$ has $M\cap M'=\HYP$.  Let $\varphi(X)$ be
  an $\Ltwo$-formula with $X$ as its only free variable, such that
  $M\models(\exists$ exactly one $X)\,\varphi(X)$, and
  $M\models\varphi(A)$.  Then $M'\models(\exists$ exactly one
  $X)\,\varphi(X)$.  Let $A'\in M'$ be such that
  $M'\models\varphi(A')$.  Then for each $n\in\omega$, we have that
  $n\in A$ if and only if $M\models\exists X\,(\varphi(X)$ and $n\in
  X)$, if and only if $M'\models\exists X\,(\varphi(X)$ and $n\in X)$,
  if and only if $n\in A'$.  Thus $A=A'$.  Hence $A\in\HYP$.  This
  completes the proof.
\end{proof}

\begin{rem}
  Theorem \ref{thm:betadef} has been announced without proof by
  Friedman \nocite{icm74}\cite[Theorem 4.3]{f-icm}.  Until now, a
  proof of Theorem \ref{thm:betadef} has not been available.
\end{rem}

We now improve Theorem \ref{thm:betadef} as follows.

Let $\le_\HYP$ denote hyperarithmetical reducibility, i.e., $X\le_\HYP
Y$ if and only if $X$ is hyperarithmetical in $Y$.

\begin{thm}\label{thm:betadef2}
  There exists a countable $\beta$-model satisfying
  \begin{quote}
    $(*)\qquad\forall X\,\forall Y\,($if $X$ is definable from
    $Y$, then $X\le_\HYP Y)$.
  \end{quote}
\end{thm}

\medskip

The $\beta$-model which we shall use to prove Theorem
\ref{thm:betadef2} is the same as for Theorem \ref{thm:betadef},
namely $M=\{G_n:n\in\omega\}$ where $\langle G_n\rangle_{n\in\omega}$
is generic.  In order to see that $M$ has the desired property, we
first relativize the proof of Theorem \ref{thm:betadef}, as follows.
Given $Y$, let $\PP^Y$ be the set of $p\in\PP$ such that there exists
$\langle X_n\rangle_{n\in\omega}$ meeting $p$ with $X_0=Y$.
(Obviously $0$ plays no special role here.)  Say that $\langle
G_n\rangle_{n\in\omega}$ is \emph{generic over} $Y$ if, for every
dense $\D^Y\subseteq\PP^Y$ definable from $Y$ over
$\HYP(Y)=\{X:X\le_\HYP Y\}$, there exists $p\in\D^Y$ such that
$\langle G_n\rangle_{n\in\omega}$ meets $p$.

\begin{lem}\label{lem:betadef}
  If $\langle G_n\rangle_{n\in\omega}$ is generic over $Y$, then
  $G_0=Y$, and $\{G_n:n\in\omega\}$ is a $\beta$-model satisfying
  $\forall X\,($if $X$ is definable from $Y$, then $X\le_\HYP Y)$.
\end{lem}

\begin{proof}
  The proof of this lemma is a straightforward relativization to $Y$
  of the proof of Theorem \ref{thm:betadef}.
\end{proof}

Consequently, in order to prove Theorem \ref{thm:betadef2}, it
suffices to prove the following lemma.

\begin{lem}\label{lem:gen2}
  If $\langle G_n\rangle_{n\in\omega}$ is generic, then $\langle
  G_n\rangle_{n\in\omega}$ is generic over $G_0$.
\end{lem}

\begin{proof}
  It suffices to show that, for all $p$ forcing $(\D^{X_0}$ is dense
  in $\PP^{X_0})$, there exists $q\le p$ forcing $\exists
  r\,(r\in\D^{X_0}$ and $\langle X_n\rangle_{n\in\omega}$ meets $r)$.
  
  Assume $p\forces(\D^{X_0}$ is dense in $\PP^{X_0})$.  Since
  $p\forces p\in\PP^{X_0}$, it follows that $p\forces\exists q\,(q\le
  p$ and $q\in\D^{X_0})$.  Fix $p'\le p$ and $q'\le p$ such that
  $p'\forces q'\in\D^{X_0}$.  Put $S'=\{X_0:\langle
  X_n\rangle_{n\in\omega}$ meets $p'\}$.  Then $S'$ is a $\Sigma^1_1$
  set, so let $e\in\omega$ be such that $S'=S_e$.  Claim 1:
  $\{(e,\langle0\rangle)\}\forces q'\in\D^{X_0}$.  If not, let
  $p''\le\{(e,\langle0\rangle)\}$ be such that $p''\forces
  q'\notin\D^{X_0}$.  Let $\pi$ be a permutation such that $\pi(0)=0$
  and $\supp(p')\cap\supp(\pi(p''))=\{0\}$.  Then
  $p'\cup\pi(p'')\in\PP$ and $\pi(p'')\forces q'\notin\D^{X_0}$, a
  contradiction.  This proves Claim 1.
  
  Claim 2: $q'\cup\{(e,\langle0\rangle)\}\in\PP$.  To see this, let
  $\langle G'_n\rangle_{n\in\omega}$ be generic meeting
  $\{(e,\langle0\rangle)\}$.  By Claim 1 we have $q'\in\D^{G'_0}$.
  Hence $q'\in\PP^{G'_0}$, i.e., there exists $\langle
  X_n\rangle_{n\in\omega}$ meeting $q'$ with $X_0=G'_0$.  Thus
  $\langle X_n\rangle_{n\in\omega}$ meets
  $q'\cup\{(e,\langle0\rangle)\}$.  This proves Claim 2.  Finally, put
  $q''=q'\cup\{(e,\langle0\rangle)\}$.  Then $q''\le q'\le p$ and
  $q''\forces(q'\in\D^{X_0}$ and $\langle X_n\rangle_{n\in\omega}$
  meets $q')$.  This proves our lemma.
\end{proof}

The proof of Theorem \ref{thm:betadef2} is immediate from Lemmas
\ref{lem:betadef} and \ref{lem:gen2}.

\section{Conservation results}
\label{sec:cons}

In this section we generalize the construction of \S\ref{sec:main} to
a wider setting.  We then use this idea to obtain some conservation
results involving the scheme $(*)$ of Theorem \ref{thm:betadef2}.

Two important subsystems of second order arithmetic are $\ATRo$
(arithmetical transfinite recursion with restricted induction) and
$\Pi^1_\infty$-$\TIo$ (the transfinite induction scheme).  For general
background, see Simpson \cite{sosoa}.  It is known \cite[\S
VII.2]{sosoa} that $\ATRo\subseteq\Pi^1_\infty$-$\TIo$, and that every
$\beta$-model is a model of $\Pi^1_\infty$-$\TIo$.

Let $(N,\mathcal{S})$ be a countable model of $\ATRo$, where
$\mathcal{S}\subseteq P(|N|)$.  Define
$\PP_{(N,\mathcal{S})}=\{p:(N,\mathcal{S})\models(p$ is a finite
subset of $\omega\times\omega^{<\omega}$ and there exists $\langle
X_n\rangle_{n\in\omega}$ meeting $p)\}$.  The notion of $\langle
G_n\rangle_{n\in|N|}$ being \emph{generic over} $(N,\mathcal{S})$ is
defined in the obvious way.  As in \S\ref{sec:main}, the basic forcing
lemmas can be proved.  Let $\langle G_n\rangle_{n\in|N|}$ be generic
over $(N,\mathcal{S})$.  Put $\mathcal{S}'=\{G_n:n\in|N|\}$.

\begin{lem}\label{lem:star}
  $(N,\mathcal{S}')$ satisfies the scheme $(*)$ of Theorem
  {\rm\ref{thm:betadef2}}.
\end{lem}

\begin{proof}
  The proof is a straightforward generalization of the arguments of
  \S\ref{sec:main}.
\end{proof}

\begin{lem}\label{lem:pres}
  Let $\psi$ be a $\Pi^1_2$ sentence with parameters from $|N|$.  If
  $(N,\mathcal{S})\models\psi$, then $(N,\mathcal{S}')\models\psi$.
\end{lem}

\begin{proof}
  Write $\psi$ as $\forall X\,(X\in S_e)$ for some fixed $e\in|N|$.
  If $(N,\mathcal{S})\models\psi$, then for each $n\in|N|$ we have
  that $\{p\in\PP_{(N,\mathcal{S})}:(e,\langle n\rangle)\in p\}$ is
  dense in $\PP_{(N,\mathcal{S})}$, hence $\emptyset\forces X_n\in
  S_e$.  Thus $\emptyset\forces\forall X\,(X\in S_e)$, i.e.,
  $\emptyset\forces\psi$, so $(N,\mathcal{S}')\models\psi$.
\end{proof}

\begin{rem}
  Since $(N,\mathcal{S})\models\ATRo$ and $\ATRo$ is axiomatized by a
  $\Pi^1_2$ sentence, it follows by Lemma \ref{lem:pres} that
  $(N,\mathcal{S}')\models\ATRo$.  Lemma \ref{lem:pres} also implies
  that $(N,\mathcal{S})$ and $(N,\mathcal{S}')$ satisfy the same
  $\Pi^1_1$ sentences with parameters from $|N|$.  From this it
  follows that the recursive well orderings of $(N,\mathcal{S})$ and
  $(N,\mathcal{S}')$ are the same, and that
  $\HYP_{(N,\mathcal{S})}=\HYP_{(N,\mathcal{S}')}$.  It can also be
  shown that $\HYP_{(N,\mathcal{S})}=\mathcal{S}\cap\mathcal{S}'$.
\end{rem}

We now have:
\begin{thm}\label{thm:atrstar}
  $\ATRo+(*)$ is conservative over $\ATRo$ for $\Sigma^1_2$ sentences.
\end{thm}

\begin{proof}
  Let $\varphi$ be a $\Sigma^1_2$ sentence.  Suppose
  $\ATRo\not\proves\varphi$.  Let $(N,\mathcal{S})$ be a countable
  model of $\ATRo+\lnot\,\varphi$.  Let $\mathcal{S}'=\{G_n:n\in|N|\}$
  as above.  Since $\lnot\,\varphi$ is a $\Pi^1_2$ sentence, we have
  by Lemmas \ref{lem:star} and \ref{lem:pres} that
  $(N,\mathcal{S}')\models\ATRo+(*)+\lnot\,\varphi$.  Thus
  $\ATRo+(*)\not\proves\varphi$.
\end{proof}

In order to obtain a similar result for $\Pi^1_\infty$-$\TIo$, we
first prove:

\begin{lem}\label{lem:ordrec}
  $(N,\mathcal{S}')\models$ ``all ordinals are recursive'', i.e.,
  ``every well ordering is isomorphic to a recursive well ordering''.
\end{lem}

\begin{proof}
  Recall from \cite[\S\S VIII.3 and VIII.6]{sosoa} that all of the
  basic results of hyperarithmetical theory are provable in $\ATRo$.
  In particular, by \cite[Theorem VIII.6.7]{sosoa}, the
  Gandy/Kreisel/Tait Theorem holds in $\ATRo$.  Thus for all
  $p\in\PP_{(N,\mathcal{S})}$ we have
  \begin{center}
    $(N,\mathcal{S})\models$ there exists $\langle X_n\rangle_{n\in\omega}$
    meeting $p$ such that $\forall
    n\,(\omega_1^{X_n}=\omega_1^{\mathrm{CK}})$.
  \end{center}
  Now, it is provable in $\ATRo$ that the predicate
  $\omega_1^X=\omega_1^{\mathrm{CK}}$ is $\Sigma^1_1$.  Thus we have
  $\emptyset\forces\forall X\,(\omega_1^X=\omega_1^{\mathrm{CK}})$,
  i.e., $\emptyset\forces$ ``all ordinals are recursive''.  This
  proves our lemma.
\end{proof}

\begin{rem}
  An alternative proof of Lemma \ref{lem:ordrec} is to note that
  $\ATRo+(*)\proves$ ``$O$ does not exist'', because $O$ would be
  definable but not hyperarithmetical.  And ``$O$ does not exist'' is
  equivalent over $\ATRo$ to ``all ordinals are recursive''.  (Here
  $O$ denotes the complete $\Pi^1_1$ set of integers.  See \cite[\S
  VIII.3]{sosoa}.)
\end{rem}

\begin{thm}\label{thm:tistar}
  $\Pi^1_\infty$-$\TIo+(*)$ is conservative over
  $\Pi^1_\infty$-$\TIo$ for $\Sigma^1_2$ sentences.
\end{thm}

\begin{proof}
  By Lemma \ref{lem:ordrec} it suffices to prove: if
  $(N,\mathcal{S})\models$ transfinite induction for recursive well
  orderings, then $(N,\mathcal{S}')\models$ transfinite induction for
  recursive well orderings.  Let $e\in|N|$ be such that
  $(N,\mathcal{S})\models$ ``${<}_e$ is a recursive well ordering''.
  Suppose $p\forces\exists n\,(n\in\field({<}_e)$ and $\varphi(n))$.
  Put $A=\{n\in\field({<}_e):p\not\forces\lnot\,\varphi(n)\}$.
  Clearly $A\ne\emptyset$.  By definability of forcing, $A$ is
  definable over $(N,\mathcal{S})$.  Hence there exists $a\in A$ such
  that $(N,\mathcal{S})\models$ ``$a$ is the ${<}_e$-least element of
  $A$''.  For each $n<_ea$ we have $p\forces\lnot\,\varphi(n)$, but
  $p\not\forces\lnot\,\varphi(a)$, so let $q\le p$ be such that
  $q\forces\varphi(a)$.  Then $q\forces$ ``$a$ is the ${<}_e$-least
  $n$ such that $\varphi(n)$''.  Thus $(N,\mathcal{S}')\models$
  transfinite induction for recursive well orderings.
\end{proof}

\begin{rem}
  Theorems \ref{thm:atrstar} and \ref{thm:tistar} are best possible,
  in the sense that we cannot replace $\Sigma^1_2$ by $\Pi^1_2$.  An
  example is the $\Pi^1_2$ sentence ``all ordinals are recursive'',
  which is provable in $\ATRo+(*)$ but not in $\Pi^1_\infty$-$\TIo$.
\end{rem}

\section{Recursion-theoretic analogs}
\label{sec:wkl}

The results of \S\S\ref{sec:main} and \ref{sec:cons} are in the realm
of hyperarithmetical theory.  We now present the analogous results in
the realm of recursion theory, concerning models of $\WKLo$.  For
background on this topic, see Simpson \cite[\S\S VIII.2 and
XI.2]{sosoa}.

\medskip

Let $\REC$ denote the set of recursive reals.  It is well known that,
for any $\omega$-model $M$ of $\WKLo$, $\REC$ is properly included in
$M$, and each $X\in\REC$ is definable in $M$.  The recursion-theoretic
analog of Theorem \ref{thm:betadef} is:

\begin{thm}\label{thm:wkldef}
  There exists a countable $\omega$-model of $\WKLo$ satisfying
  \begin{center}
    $\forall X\,($if $X$ is definable, then $X\in\REC)$.
  \end{center}
\end{thm}

\begin{proof}
  Use exactly the same construction and argument as for Theorem
  \ref{thm:betadef}, replacing $\Sigma^1_1$ sets by $\Pi^0_1$ subsets
  of $2^\omega$.
\end{proof}

\begin{rem}
  Theorem \ref{thm:wkldef} is originally due to Friedman \cite[Theorem
  1.10, unpublished]{f-sosoa} (see also \cite[Theorem 1.6]{f-icm}).
  It was later proved again by Simpson \nocite{rm2001}\cite{pizowkl}
  (see also Simp\-son/Ta\-na\-ka/Ya\-ma\-zaki \cite{sty}).  All three
  proofs of Theorem \ref{thm:wkldef} are different from one another.
\end{rem}

Let $\le_T$ denote Turing reducibility, i.e., $X\le_TY$ if and only if
$X$ is computable using $Y$ as an oracle.  The recursion-theoretic
analog of Theorem \ref{thm:betadef2} is:

\begin{thm}\label{thm:wkldef2}
  There exists a countable $\omega$-model of $\WKLo$ satisfying
  \begin{center}
    $(**)\qquad\forall X\,\forall Y\,($if $X$ is definable from
    $Y$, then $X\le_TY)$.
  \end{center}
\end{thm}

\begin{proof}
  Use exactly the same construction and argument as for Theorem
  \ref{thm:betadef2}, replacing $\Sigma^1_1$ sets by $\Pi^0_1$ subsets
  of $2^\omega$.
\end{proof}

The recursion-theoretic analog of Theorems \ref{thm:atrstar} and
\ref{thm:tistar} is:

\begin{thm}\label{thm:wklstar}
  $\WKLo+(**)$ is conservative over $\WKLo$ for arithmetical sentences.
\end{thm}

\begin{proof}
  The proof is analogous to the arguments of \S\ref{sec:cons}.  
\end{proof}

\begin{rem}
  Theorems \ref{thm:wkldef2} and \ref{thm:wklstar} are originally due
  to Simpson \cite{pizowkl}.  The proofs given here are different from
  the proofs that were given in \cite{pizowkl}.
\end{rem}

\section{Some additional results}
\label{sec:friedman}

In this section we present some additional results and open questions.

\begin{lem}\label{lem:s11p}
  Let $\varphi(X)$ be a $\Sigma^1_1$ formula with no free set
  variables other than $X$.  The following is provable in $\ATRo$.  If
  $\exists X\,(X\notin\HYP$ and $\varphi(X))$, then $\exists P\,(P$ is
  a perfect tree and $\forall X\,($if $X$ is a path through $P$ then
  $\varphi(X)))$.
\end{lem}

\begin{proof}
  This is a well known consequence of formalizing the Perfect Set
  Theorem within $\ATRo$.  See Simpson \cite[\S\S V.4 and
  VIII.3]{sosoa}.  See also Sacks \cite[\S III.6]{sacks-hrt}.
\end{proof}

\begin{lem}\label{lem:s12p}
  Let $\varphi(X)$ be a $\Sigma^1_2$ formula with no free set
  variables other than $X$.  The following is provable in $\ATRo\,+{}$
  ``all ordinals are recursive''.  If $\exists X\,(X\notin\HYP$ and
  $\varphi(X))$, then $\exists P\,(P$ is a perfect tree and $\forall
  X\,($if $X$ is a path through $P$ then $\varphi(X)))$.
\end{lem}

\begin{proof}
  Since $\varphi(X)$ is $\Sigma^1_2$, we can write
  $\varphi(X)\equiv\exists Y\,\forall f\,\exists n\,R(X[n],Y[n],f[n])$
  where $R$ is a primitive recursive predicate.  Let $T_{X,Y}$ be the
  tree consisting of all $\tau$ such that $\lnot\,(\exists
  n\le\lh(\tau))\,R(X[n],Y[n],\tau[n])$.  For
  $\alpha<\omega_1^{\mathrm{CK}}$ put
  \begin{center}
    $\varphi_\alpha(X)\,\,\equiv\,\,\exists Y\,(T_{X,Y}$ is well
    founded of height $\le\alpha)$.
  \end{center}
  Note that, for each $\alpha<\omega_1^{\mathrm{CK}}$,
  $\varphi_\alpha(X)$ is $\Sigma^1_1$.  Reasoning in $\ATRo\,+{}$
  ``all ordinals are recursive'', we have $\forall X\,(\varphi(X)$ if
  and only if
  $(\exists\alpha<\omega_1^{\mathrm{CK}})\,\varphi_\alpha(X))$.  Thus
  Lemma \ref{lem:s12p} follows easily from Lemma \ref{lem:s11p}.
\end{proof}

\begin{thm}\label{thm:s12p}
  Let $T$ be $\ATRo$ or $\Pi^1_\infty$-$\TIo$.  Let $\varphi(X)$ be a
  $\Sigma^1_2$ formula with no free set variables other than $X$.  If
  $T\proves\exists X\,(X\notin\HYP$ and $\varphi(X))$, then
  $T\proves\exists P\,(P$ is a perfect tree and $\forall X\,($if $X$
  is a path through $P$ then $\varphi(X)))$.
\end{thm}

\begin{proof}
  From Friedman \cite{f-bi} or Simpson \cite[\S VII.2]{sosoa}, we have
  that $T\proves$ the disjunction (1) all ordinals are recursive, or
  (2) there exists a countable coded $\beta$-model $M$ satisfying
  $T\,+{}$ ``all ordinals are recursive''.  In case (1), the desired
  conclusion follows from Lemma \ref{lem:s12p}.  In case (2), we have
  $M\models\exists X\,(X\notin\HYP$ and $\varphi(X))$, so the proof of
  Lemma \ref{lem:s12p} within $M$ gives
  $\alpha<\omega_1^{\mathrm{CK}}$ such that $M\models\exists
  X\,(X\notin\HYP$ and $\varphi_\alpha(X))$.  It follows that $\exists
  X\,(X\notin\HYP$ and $\varphi_\alpha(X))$.  We can then apply Lemma
  \ref{lem:s11p} to $\varphi_\alpha(X)$ to obtain the desired
  conclusion.
\end{proof}

\begin{cor}\label{cor:s12p}
  Let $T$ and $\varphi(X)$ be as in Theorem {\rm\ref{thm:s12p}}.  If
  $T\proves\exists X\,(X\notin\HYP$ and $\varphi(X))$, then
  $T\proves\forall Y\,\exists X\,(\varphi(X)$ and $\forall
  n\,(X\ne(Y)_n))$.
\end{cor}

\begin{proof}
  This follows easily from Theorem \ref{thm:s12p}.
\end{proof}

\begin{thm}\label{thm:s12}
  Let $T$ and $\varphi(X)$ be as in Theorem {\rm\ref{thm:s12p}}.  If
  $T\proves(\exists$ exactly one $X)\,\varphi(X)$, then
  $T\proves\exists X\,(X\in\HYP$ and $\varphi(X))$.
\end{thm}

\begin{proof}
  Consider cases (1) and (2) as in the proof of Theorem
  \ref{thm:s12p}.  In both cases it suffices to show that, for all
  $\alpha<\omega_1^{\mathrm{CK}}$, if $(\exists$ exactly one
  $X)\,\varphi_\alpha(X)$ then $\exists X\,(X\in\HYP$ and
  $\varphi_\alpha(X))$.  This follows from Lemma \ref{lem:s11p}
  applied to $\varphi_\alpha(X)$.
\end{proof}

\begin{rem}
  Theorems \ref{thm:s12p} and \ref{thm:s12} appear to be new.
  Corollary \ref{cor:s12p} has been stated without proof by Friedman
  \cite[Theorems 3.4 and 4.4]{f-icm}.  A recursion-theoretic analog of
  Corollary \ref{cor:s12p} has been stated without proof by Friedman
  \cite[Theorem 1.7]{f-icm}, but this statement of Friedman is known
  to be false, in view of Simpson \cite{pizowkl}.  A
  recursion-theoretic analog of Theorem \ref{thm:s12} has been proved
  by Simpson/Tanaka/Yamazaki \cite{sty}.
\end{rem}

\begin{que}
  Suppose $\WKLo\proves\exists X\,(X\notin\REC$ and $\varphi(X))$
  where $\varphi(X)$ is $\Sigma^1_1$, or even arithmetical, with no
  free set variables other than $X$.  Does it follow that
  $\WKLo\proves\exists X\,\exists Y\,(X\ne
  Y\land\varphi(X)\land\varphi(Y))$?  A similar question has been
  asked by Friedman \cite[unpublished]{f-sosoa}.
\end{que}

\begin{que}
  Suppose $\WKLo\proves(\exists$ exactly one $X)\,\varphi(X)$ where
  $\varphi(X)$ is $\Sigma^1_1$ with no free set variables other than
  $X$.  Does it follow that $\WKLo\proves\exists X\,(X\in\REC$ and
  $\varphi(X))$?  If $\varphi(X)$ is arithmetical or $\Pi^1_1$ then
  the answer is yes, by Simp\-son/Ta\-naka/Yama\-zaki \cite{sty}.
\end{que}


\end{document}